\documentclass[11pt]{amsart}
\usepackage[pagewise]{lineno}
\usepackage{amsmath,amssymb,latexsym,soul,cite,mathrsfs}

\usepackage{color,enumitem,graphicx}
\usepackage[colorlinks=true,urlcolor=blue,
citecolor=red,linkcolor=red,linktocpage,pdfpagelabels,
bookmarksnumbered,bookmarksopen]{hyperref}
\usepackage[english]{babel}

\usepackage[left=2.9cm,right=2.9cm,top=2.8cm,bottom=2.8cm]{geometry}
\usepackage[hyperpageref]{backref}

\usepackage[colorinlistoftodos]{todonotes}
\makeatletter
\providecommand\@dotsep{5}
\def\listtodoname{List of Todos}
\def\listoftodos{\@starttoc{tdo}\listtodoname}
\makeatother

\numberwithin{equation}{section}

\newcommand{\Om} {\Omega}
\newcommand{\la} {\lambda}
\newcommand{\La} {\Lambda}
\newcommand{\mb} {\mathbb}
\newcommand{\mc} {\mathcal}
\newcommand{\noi} {\noindent}
\newcommand{\uline} {\underline}
\newcommand{\oline} {\overline}

\newtheorem{Theorem}{Theorem}[section]
\newtheorem{Lemma}[Theorem]{Lemma}

\newtheorem{Remark}[Theorem]{Remark}
\newtheorem{Definition}[Theorem]{Definition}

\begin{document}

\title[Weighted singular anisotropic equations]
{On a class of weighted anisotropic $p$-Laplace equation with singular nonlinearity}
\author{Prashanta Garain}


\address[Prashanta Garain ]
{\newline\indent Department of Mathematics
	\newline\indent
Uppsala University
	\newline\indent
S-751 06 Uppsala, Sweden
\newline\indent
Email: {\tt pgarain92@gmail.com} }

\pretolerance10000


\begin{abstract}
We consider a class of singular weighted anisotropic $p$-Laplace equations. We provide sufficient condition on the weight function that may vanish or blow up near the origin to ensure the existence of at least one weak solution in the purely singular case and at least two different weak solutions in the purturbed singular case.
\end{abstract}

\maketitle

\section{Introduction}
In this article, we establish existence of weak solutions for the following class of weighted anisotropic singular problems
\begin{equation}\label{maineqn}
-F_{p,w}u=g(x,u)\text{ in }\Om,\quad u>0\text{ in }\Om,\quad u=0\text{ on }\partial\Om,
\end{equation}
where $1<p<\infty$, $\Om\subset\mathbb{R}^N$ is a bounded smooth domain with $N\geq 2$. Here 
$$
F_{p,w}u:=\text{div}(w(x)F(\nabla u)^{p-1}\nabla_{\xi}F(\nabla u))
$$
is the weighted anisotropic $p$-Laplace operator, where $F:\mathbb{R}^N\to[0,\infty)$ is the Finsler-Minkowski norm, that is
\begin{enumerate}
\item[(H0)] $F(x)\geq 0$, for every $x\in\mathbb{R}^N$.
\item[(H1)] $F(x)=0$, if and only if $x=0$.
\item[(H2)] $F(tx)=|t|H(x)$, for every $x\in\mathbb{R}^N$ and $t\in\mathbb{R}$.
\item[(H3)] $F\in C^{\infty}\left(\mathbb{R}^N\setminus\{0\}\right)$.
\item[(H4)] the Hessian matrix $\nabla_{\xi}^2\big(\frac{F^2}{2}\big)(x)$ is positive definite for all $x\in\mathbb{R}^N\setminus\{0\}$.
\end{enumerate} 
Here $\nabla_{\xi}$ denotes the gradient operator with respect to the $\xi$ variable. The weight function $w$ belong to a class of $p$-admissible weights $W_p^{s}$ defined in Section $2$. By singularity, we refer to the blow up property of the nonlinearity $g$ in the right hand side of \eqref{maineqn}. We discuss existence of at least one weak solution of the problem \eqref{maineqn} for the purely singular nonlinearity $g$ of the form $(g_1)$ given by
$$
g(x,u)=\la h(u)u^{-\gamma},\quad \leqno(g_1)
$$
where $\la>0,\gamma\in(0,1)$ and 
\begin{enumerate}
\item[$(h_1)$] $h:[0,\infty)\to\mathbb{R}$ is a continuous nondecreasing function such that $h(0)>0$ and
\item[$(h_2)$]
\begin{equation}\label{H1f1}
\lim_{t\to 0}\frac{h(t)}{t^\gamma}=\infty, \quad \lim_{t\to\infty}\frac{h(t)}{t^{\gamma+p-1}}=0.
\end{equation}
\end{enumerate}
Further, we study multiplicity result of the equation \eqref{maineqn} for the purturbed singular nonlinearity $g$ of the form $(g_2)$ given by
$$
g(x,u)=\la u^{-\gamma}+u^q,\quad\leqno(g_2)
$$
for a certain range of $\la,q>0$, when $\gamma\in(0,1)$.

To give some more insight on the equation \eqref{maineqn}, let us discuss few examples of $F$.

\noi \textbf{Examples:} Let $x=(x_1,x_2,\ldots,x_N)\in\mathbb{R}^N$.
\begin{enumerate}
\item[(i)] For $t>1$, we define 
\begin{equation}\label{ex1}
F_t(x):=\Big(\sum_{i=1}^{N}|x_i|^t\Big)^\frac{1}{t}.
\end{equation}
\item[(ii)] For $\lambda,\mu>0$, we define
\begin{equation}\label{ex2}
F_{\lambda,\mu}(x):=\sqrt{\lambda\sqrt{\sum_{i=1}^{N}x_i^{4}}+\mu\sum_{i=1}^{N}x_i^{2}}.
\end{equation}
\end{enumerate}
Then, it follows that the functions $F_t, F_{\lambda,\mu}:\mathbb{R}^N\to[0,\infty)$ given by \eqref{ex1} and \eqref{ex2} satisfies all the hypothesis from (H0)-(H4), see Mezei-Vas \cite{MV}.

\begin{Remark}\label{exrmk1}
For $i=1,2$, if $\lambda_i,\mu_i$ are positive real numbers such that $\frac{\lambda_1}{\mu_1}\neq\frac{\lambda_2}{\mu_2}$, then $F_{\lambda_1,\mu_1}$ and $F_{\lambda_2,\mu_2}$ given by \eqref{ex2} defines two non-isometric norms in $\mathbb{R}^N$.
\end{Remark}

\begin{Remark}\label{exrmk2}
If $F=F_t$ is given by \eqref{ex1}, then we have
\begin{equation}\label{ex}
F_{p,w}u=
\begin{cases}
\Delta_{p,w} u:=\text{div}(w(x)|\nabla u|^{p-2}\nabla u),\, (weighted\,$p$\text{-Laplacian})\,\text{if }t=2,\,1<p<\infty,\\
S_{p,w}u=\sum_{i=1}^{N}\frac{\partial}{\partial x_i}\Big(|u_i|^{p-2}u_i\Big),\,(weighted\,pseudo\, $p$\text{-Laplacian})\,\text{if }t=p\in(1,\infty),
\end{cases}
\end{equation}
where $u_i:=\frac{\partial u}{\partial x_i}$, for $i=1,2,\ldots,N$.
\end{Remark}

Therefore, it is clear from the above examples that the equation \eqref{maineqn} covers a wide range of weighted singular equations. In particular, equation \eqref{maineqn} extends the weighted $p$-Laplace equation
\begin{equation}\label{maineqn1}
-\Delta_{p,w}u=g(x,u)\text{ in }\Om,\quad u>0\text{ in }\Om,\quad u=0\text{ on }\partial\Om.
\end{equation}
For bounded weight function $w$, the equation \eqref{maineqn1} is investigated thoroughly over the last three decade for various type of singular nonlinearity $g$. We refer to Crandall-Rabinowitz-Tartar \cite{CRT}, Ghergu-R\u{a}dulescu \cite{book-radu}, Boccardo-Orsina \cite{Boc} and the references therein for the singular Laplace equation. Further, the nonlinear case is studied as well. In this concern, we refer to Canino-Sciunzi-Trombetta \cite{Canino}, De Cave \cite{DeCave}, Oliva-Orsina-Petitta \cite{PetOr, Petitta} and the references therein.

When $g$ is of the form $(g_1)$ and $w=1$, equation \eqref{maineqn1} is studied by Ko-Lee-Shivaji \cite{KLS} to obtain existence results. For the purturbed case $(g_2)$ and $w=1$, equation \eqref{maineqn1} is studied by Arcoya-Boccardo-Me\'rida \cite{arcoyaBoc, Merida} to obtain multiplicity results for a certain range of the parameters $\la,q>0$. These results has been further extended to investigate the $p$-Laplace equation
\begin{equation}\label{plap}
-\text{div}(|\nabla u|^{p-2}\nabla u)=\la u^{-\gamma}+u^q\text{ in }\Om,\quad u>0\text{ in }\Om,\quad u=0\text{ on }\partial\Om.
\end{equation}
See Giacomoni-Schindler-Tak\'{a}\v{c} \cite{GST}, Bal-Garain \cite{BG} and the references therein.

Anisotropic $p$-Laplace equation has been a topic of considerable attention in the recent years. In the nonsingular unweighted case, we refer to Alvino-Ferone-Trombetti-Lions \cite{AFTL}, Belloni-Ferone-Kawohl \cite{BFKzamp}, Belloni-Kawohl-Juutinen \cite{BKJ}, Bianchini-Giulio \cite{BC}, Xia \cite{Xiathesis},  Cianchi-Salani\cite{CS}, Ferone-Kawohl \cite{FK}, Kawohl-Novaga \cite{KN} and with weights, see Dipierro-Poggesi-Valdinoci \cite{Val}. When $g$ is singular, for $w=1$, anisotropic $p$-Laplace equations is investigated by Biset-Mebrate-Mohammed \cite{BMM20}, Farkas-Winkert \cite{PF20} and Farkas-Fiscella-Winkert \cite{PF21}, Bal-Garain-Mukherjee \cite{BGM}.

Although it is worth mentioning that singular weighted anisotropic equations are very less understood, when the weight function $w$ vanish or blow up near the origin (for example, $w(x)=|x|^{\alpha},\,\alpha\in\mathbb{R}$). Such weighted equations are referred to as degenerate equation, where the degeneracy is captured by the weight function $w$. In the nonsingular setting, such a situation is discussed for weighted $p$-Laplace equations in detail in the literature, refer to Dr\'abek-Kufner-Nicolosi \cite{Drabek}, Heinonen-Kilpel\"ainen-Martio \cite{Juh} and the references therein. In the singular case, recently, for a class of Muckenhoupt weights, existence results for the weighted $p$-Laplace equation \eqref{maineqn1} with various type of singularity $g$ have been investigated in Garain-Mukherjee \cite{Gmn, GMmed}. Further, these class of weights has been generalized to a class of $p$-admissible weights for the weighted $p$-Laplace equation \eqref{maineqn1}. For such weights, nonexistence results has been discussed in Garain-Kinnunen \cite{GKpams} and existence results also established in Hara \cite{Hara} for the purely singular nonlinearity $g$. The weighted anisotropic case is recently discussed in Bal-Garain \cite{Gpadm, BGpadm} for a class of $p$-admissible weights.

In this article, our main purpose is to provide sufficient condition on the weight function $w$ to ensure the existence results for the weighted anisotropic $p$-Laplace equation \eqref{maineqn}. More precisely, when $g$ takes the purely singular form $(g_1)$, we prove existence of at least one weak solution of the problem \eqref{maineqn} (see Theorem \ref{pu2thm}) and existence of at least two different weak solutions of the problem \eqref{maineqn}, when $g$ takes the purturbed singular form $(g_2)$ (Theorem \ref{mthm}).

Some major difficulties in the weighted case are that suitable embedding results and regularity results are not readily available. We found a class of $p$-admissible weights $W_p^{s}$ defined in Section $2$ to be useful, which allows us to shift from the weighted Sobolev spaces into the unweighted Sobolev spaces.

To deal with the nonlinearity $(g_1)$, we follow the approach from Ko-Lee-Shivaji \cite{KLS}. To this end, we construct suitable subsolution and supersolution of \eqref{maineqn} by implementing the idea from Haitao \cite{YHaitao} to the weighted case.

To deal with $(g_2)$, we employ the variational approach from Arcoya-Boccardo \cite{arcoyaBoc}. To this end, we also need some existence and regularity properties of the eigenfunction of the weighted anisotropic eigenvalue problem
\begin{equation}\label{wevp}
-F_{p,w}u=\la|u|^{p-2}u \text{ in } \Om,\,\,u=0\text{ on }\partial\Om.
\end{equation}
Such type of equation is studied for the unweighted setting in Belloni-Ferone-Kawohl \cite{BFKzamp}. As far as we are aware, the weighted case is unknown. Although the proof is analogous to Dr\'abek-Kufner-Nicolosi \cite{Drabek}, for the readers convenience, we give a proof in the appendix Section $5$.

This article is organized as follows: In Section $2$, we discuss some preliminary results and state our main results. In Sections $3-4$, we prove our main results. Finally, in the appendix Section $5$, we study the weighted anisotropic eigenvalue problem \eqref{wevp}.

\section{Preliminaries}
Throughout the rest of the article, we assume $1<p<\infty$, unless otherwise mentioned.

In order to define the notion of weak solutions for the problem \eqref{maineqn}, first we present some known facts about weighted Sobolev spaces, refer to \cite{Juh} for more details.

Let $1<p<\infty$. We say that $w$ is a $p$-admissible weight $W_p$, if $w\in L^1_{\mathrm{loc}}(\mathbb{R}^N)$ such that $0<w<\infty$ almost everywhere in $\mathbb{R}^N$ and satisfies the following assumptions:
\begin{enumerate}
\item[(i)] for any ball $B$ in $\mathbb{R}^N$, there exists a positive constant $C_{\mu}$ such that $\mu(2B)\leq C_{\mu} \mu(B)$, where $\mu(E)=\int_{E}w\,dx$ for a measurable subset $E$ in $\mathbb{R}^N$ and $d\mu(x)=w(x)\,dx$.
\item[(ii)] If $D$ is an open set and $\phi\in C^{\infty}(D)$ is a sequence of functions such that $\int_{D}|\phi_i|^p\,d\mu\to 0$ and $\int_{D}|\nabla\phi_i-v|^p\,d\mu\to 0$ as $i\to\infty$, where $v$ is a vector valued measurable function in $L^p(D,w)$, then $v=0$.
\item[(iii)] There exists constants $\kappa>1$ and $C_1>0$ such that 
\begin{equation}\label{wp}
\left(\frac{1}{\mu(B)}\int_{B}|\phi|^{\kappa p}\,d\mu\right)^\frac{1}{\kappa p}\leq C_1 r \left(\frac{1}{\mu(B)}\int_{B}|\nabla\phi|^p\,d\mu\right)^\frac{1}{p},
\end{equation}
whenever $B=B(x_0,r)$ is a ball in $\mathbb{R}^N$ and $\phi\in C_{c}^\infty(B)$.
\item[(iv)] There exists a positive constant $C_2$ such that
\begin{equation}\label{wp1}
\int_{B}|\phi-\phi_B|^p\,d\mu\leq C_2 r^p\int_{B}|\nabla\phi|^p\,d\mu,
\end{equation}
whenever $B=B(x_0,r)$ is a ball in $\mathbb{R}^N$ and $\phi\in C^\infty(B)$ is bounded. Here
$$
\phi_B=\frac{1}{\mu(B)}\int_{B}\phi\,d\mu.
$$
\end{enumerate}

\textbf{Examples:} The class of Muckenhoupt weights $A_p$ are $p$-admissible, see \cite[Theorem 15.21]{Juh}. In particular, if $c\leq w\leq d$ for some positive constants $c,d$, then $w\in A_p$ for any $1<p<\infty$. Let $1<p<N$ and $J_f(x)$ denote the determinant of the Jacobian matrix of a $K$-quasiconformal mapping $f:\mathbb{R}^N\to\mathbb{R}^N$, then $w(x)=J_f(x)^{1-\frac{p}{N}}\in W_q$ for any $q\geq p$, see \cite[Corollary 15.34]{Juh}. If $1<p<\infty$ and $\nu>-N$, then $w(x)=|x|^{\nu}\in W_p$, see \cite[Corollary 15.35]{Juh}. For more examples, refer to \cite{ex1, ex2, ex3, Tero, Juh} and the references therein.

\begin{Definition}(Weighted Spaces)
Let $1<p<\infty$ and $w\in W_p$. Then the weighted Lebesgue space $L^{p}(\Omega,w)$ is the class of measurable functions $u:\Om\to\mathbb{R}$ such that the norm of $u$ given by
\begin{equation}\label{lnorm}
\|u\|_{L^p(\Om,w)} = \Big(\int_{\Omega}|u(x)|^{p} w(x)\,dx\Big)^\frac{1}{p}<\infty.
\end{equation}
The weighted Sobolev space $W^{1,p}(\Om,w)$ is the class of measurable functions $u:\Om\to\mathbb{R}$ such that
\begin{equation}\label{norm1}
\|u\|_{1,p,w} = \Big(\int_{\Omega}|u(x)|^{p} w(x)\,dx+\int_{\Omega}|\nabla u(x)|^{p} w(x)\,dx\Big)^\frac{1}{p}<\infty.
\end{equation}
If $u\in W^{1,p}(\Om',w)$ for every $\Om'\Subset\Om$, then we say that $u\in W^{1,p}_{\mathrm{loc}}(\Om,w)$. The weighted Sobolev space with zero boundary value is defined as
$$
W^{1,p}_{0}(\Omega,w)=\overline{\big(C_{c}^{\infty}(\Omega),\|\cdot\|_{1,p,w}\big)}.
$$
\end{Definition}
By the Poincar\'e inequality from \cite{Juh} and Remark \ref{hypormk2} below, the norm defined by \eqref{norm1} on the space $W_{0}^{1,p}(\Omega,w)$ is equivalent to the norm given by
\begin{equation}\label{equinorm}
\|u\|_{W_0^{1,p}(\Om,w)}=\Big(\int_{\Omega}F(\nabla u)^p w\,dx\Big)^\frac{1}{p}.
\end{equation}
Moreover, the space $W^{1,p}_{0}(\Omega,w)$ is a separable and uniformly convex Banach space, see \cite{Juh}.

Next, we state an embedding result for a sublcass of $W_p$, which is important in our context. To this end, let
\begin{equation}\label{I}
I:=\Big[\frac{1}{p-1},\infty\Big)\cap\Big(\frac{N}{p},\infty\Big)
\end{equation}
and consider the following subclass of $W_p$ given by
\begin{equation}\label{wgtcls}
W_p^{s} := \Big\{w\in W_p: w^{-s}\in L^{1}(\Omega)\,\,\text{for some}\,\,s\in I\Big\}.
\end{equation}
Then we observe that
$$
w(x)=|x|^\alpha\in W_{p}^s,\text{ for }-N<\alpha<\min\Big\{\frac{N}{s},N(p-1)\Big\}.
$$
We refer the interested reader to \cite{Drabek} for more applications of such class of weight functions. Following the lines of the proof of \cite[Theorem $2.6$]{Gmn} based on \cite{Drabek} the following embedding result holds.
\begin{Lemma}\label{emb}
Let $w \in W_p^{s}$ for some $s\in I$. Then the following continuous inclusion maps hold
\[
    W^{1,p}(\Omega,w)\hookrightarrow W^{1,p_s}(\Omega)\hookrightarrow 
\begin{cases}
    L^t(\Omega),& \text{for } p_s\leq t\leq p_s^{*}, \text{ if } 1\leq p_s<N, \\
    L^t(\Omega),& \text{ for } 1\leq t< \infty, \text{ if } p_s=N, \\
    C(\overline{\Omega}),& \text{ if } p_s>N,
\end{cases}
\]
where $p_s = \frac{ps}{s+1} \in [1,p)$.
Moreover, the second embedding above is compact except for $t=p_s^{*}=\frac{Np_s}{N-p_s}$, if $1\leq p_s<N$ and the same result holds for the space $W_{0}^{1,p}(\Omega,w)$.
\end{Lemma}

\begin{Remark}\label{Embrmk}
We note that if $0<c\leq w\leq d$ for some constants $c,d$, then $W^{1,p}(\Omega,w)=W^{1,p}(\Omega)$ and by the Sobolev embedding, Lemma \ref{emb} holds by replacing $p_s$ and $p_s^{*}$ with $p$ and $p^{*}$ respectively.
\end{Remark}


Now we are ready to define the notion of weak solutions for the problem \eqref{maineqn}. 

\begin{Definition}\label{wksoldef}
Let $w\in W_p^{s}$ and $g$ be either of the form $(g_1)$ or $(g_2)$. We say that $u\in W_0^{1,p}(\Omega,w)$ is a weak subsolution (or supersolution) of \eqref{maineqn}, if $u>0$ in $\Omega$ such that for every $\omega\Subset\Omega$, there exists a positive constant $c(\omega)$ with $u\geq c(\omega)>0$ in $\omega$ and
\begin{equation}\label{wksoleqn}
\int_{\Omega}w(x)F(\nabla u)^{p-1}\nabla_{\xi}F(\nabla u)\nabla\phi\,dx\leq (\text{ or })\geq g(x,u)\phi\,dx,
\end{equation}
for every nonnegative $\phi\in C_c^{1}(\Omega)$. We say that $u\in W_0^{1,p}(\Omega,w)$ is a weak solution of \eqref{maineqn}, if the equality in \eqref{wksoleqn} holds for every $\phi\in C_c^{1}(\Omega)$ without a sign restriction.
\end{Definition}

\begin{Remark}\label{tstrmk}
If $u$ is a weak solution of \eqref{maineqn}, then proceeding similarly as in the proof of \cite[Lemma $2.13$]{BGM}, it follows that the equality in \eqref{wksoleqn} holds for every $\phi\in W_0^{1,p}(\Om,w)$.
\end{Remark}

\textbf{Statement of the main results:} Our main results in this article reads as follows:

\begin{Theorem}\label{pu2thm}
Let $2\leq p<\infty$, $0<\gamma<1$ and $w\in W_p^{s}$ for some $s\in I$. Assume that $g$ is of the form $(g_1)$. Then for every $\lambda>0$, there exists a weak solution $u\in W_0^{1,p}(\Omega,w)\cap L^{\infty}(\Omega)$ of the problem \eqref{maineqn}.
\end{Theorem}

\begin{Theorem}\label{pu2thm2}
If $F_{p,w}=\Delta_{p,w}$ or $S_{p,w}$ given by \eqref{ex}, then Theorem \ref{pu2thm} holds for any $1<p<\infty$.
\end{Theorem}

\begin{Remark}\label{pu2thm2rmk}
If $F_{p,w}=\Delta_{p,w}$, then Theorem \ref{pu2thm2} extends \cite[Theorem 2.10]{GMmed} to the class of $p$-admissible weights $W_p^{s}$.  
\end{Remark}

\begin{Theorem}\label{mthm}
Let $2\leq p<\infty$, $0<\gamma<1$ and $g$ is of the form $(g_2)$. Suppose that $w\in W_p^{s}$ for some $s\in I$ and $q\in(p-1,p_s^{*}-1)$, where $p_s^{*}=\frac{Np_{s}}{N-p_{s}}$ if $1\leq p_s<N$, $p_s^{*}=\infty$ if $p_s\geq N$. Then there exists $\Lambda>0$ such that for every $\lambda\in(0,\Lambda)$ the problem \eqref{psp} has at least two different weak solutions in $W_0^{1,p}(\Omega,w)$. 
\end{Theorem}

\begin{Remark}\label{wrmk2}
If $c\leq w\leq d$ for some positive constants $c,d$, then noting Remark \ref{Embrmk} and arguing similarly, Theorem \ref{mthm} will hold by replacing $p_s$ and $p_s^{*}$ with $p$ and $p^{*}$ respectively.
\end{Remark}

\begin{Theorem}\label{mthm2}
If $F_{p,w}=\Delta_{p,w}$ or $S_{p,w}$ given by \eqref{ex}, then Theorem \ref{mthm} holds for any $1<p<\infty$.
\end{Theorem}

\begin{Remark}\label{mthm2rmk}
If $F_{p,w}=\Delta_{p,w}$, then Theorem \ref{mthm2} extends \cite[Theorem 2.11]{GMmed} to any $1<p<\infty$ and to the class of $p$-admissible weights $W_p^{s}$. 
\end{Remark}

\textbf{Auxiliary results:} We end this section by stating some auxiliary results. First we note the following remark.

\begin{Remark}\label{hypormk2}
Since all norms in $\mathbb{R}^N$ are equivalent, there exists positive constants $C_1, C_2$ such that
$$
C_1|x|\leq F(x)\leq C_2|x|,\quad\forall \,x\in\mathbb{R}^N.
$$
\end{Remark}

For some positive constants $c_1, c_2$, the next result follows from \cite[Corollary 2.2]{BGM}.

\begin{Lemma}\label{regrmk}
The following relations hold:
\begin{equation}\label{lbd}
F(x)^{p-1}\nabla_{\xi}F(x)\cdot x=F(x)^p\geq c_1|x|^p,\quad\forall\,x\in\mathbb{R}^N,
\end{equation}
\begin{equation}\label{ubd}
\big|F(x)^{p-1}\nabla_{\xi}F(x)\big|\leq c_2|x|^{p-1},\quad\forall\,x\in\mathbb{R}^N\text{ and}
\end{equation}
\begin{equation}\label{homo}
F(tx)^{p-1}\nabla_{\xi}F(tx)=|t|^{p-2}t F(x)^{p-1}\nabla_{\xi}F(x),\quad\forall\,x\in\mathbb{R}^N\text{ and }t\in\mathbb{R}\setminus\{0\}.
\end{equation}
\end{Lemma}


Also, we have the following anisotropic algebraic inequality from \cite[Lemma 2.5]{BGM}.

\begin{Lemma}\label{alg}
Let $2\leq p<\infty$. Then, for every $x,y\in\mathbb{R}^N$, there exists a positive constant $c$ such that
\begin{equation}\label{algineq}
\begin{split}
\langle F(x)^{p-1}\nabla_{\xi}F(x)-F(y)^{p-1}\nabla_{\xi}F(y),x-y\rangle&\geq
cF(x-y)^p.
\end{split}
\end{equation}
\end{Lemma}

Next, we state the algebraic inequality from Peral \cite[Lemma A.0.5]{Pe}.

\begin{Lemma}\label{AI}
For any $a,b\in\mathbb{R}^N$, there exists a positive constant $C$, such that
\begin{equation}\label{ALGin}
\langle |a|^{p-2}a-|b|^{p-2}b, a-b \rangle\geq
\begin{cases}
C|a-b|^p,\text{ if }2\leq p<\infty,\\
C\frac{|a-b|^2}{(|a|+|b|)^{2-p}},\text{ if }1<p<2.
\end{cases}
\end{equation}
\end{Lemma}

We have the following result from \cite[Lemma $B.1$]{Gstam}.

\begin{Lemma}\label{useful for L infinity estimate}
Let $\phi(t),\,k_0\leq t < \infty,$ be nonnegative and nonincreasing such that
$$
\phi(h)\leq\frac{C}{(h-k)^l}\phi(k)^m,\;\;h > k > k_0,
$$
where $C,l,m$ are positive constants with $m > 1$. Then
$
\phi(k_0+d) = 0,
$
where
$$
d^l = C\phi(k_0)^{m-1}2^\frac{lm}{m-1}.
$$
\end{Lemma}

\textbf{Notation:} Throughout the rest of the article, we shall use the following notations.
\begin{itemize}

\item For $u\in W_0^{1,p}(\Om,w)$, denote by $\|u\|$ to mean the norm $\|u\|_{W_0^{1,p}(\Om,w)}$ as defined by \eqref{equinorm}.

\item For given constants $c,d$ and a set $S$, by $c\leq u\leq d$ in $S$, we mean $c\leq u\leq d$ almost everywhere in $S$. Moreover, we write $|S|$ to denote the Lebesgue measure of $S$.

\item $\langle,\rangle$ denotes the standard inner product in $\mathbb{R}^N$.

\item The conjugate exponent of $\theta>1$ by $\theta':=\frac{\theta}{\theta-1}$.

\item We denote by $p^*:=\frac{Np}{N-p}$ if $1<p<N$, $p^*=\infty$ if $p\geq N$. Analogously, we denote by $p_{s}^*:=\frac{Np_{s}}{N-p_{s}}$ if $1\leq p_s<N$, $p_{s}^*=\infty$ if $p_s\geq N$.

\item For $a\in\mathbb{R}$, we denote by $a^+:=\max\{a,0\}$, $a^-:=\max\{-a,0\}$ and $a_-:=\min\{a,0\}$.

\item We write by $c,C$ or $C_i$ for $i\in\mathbb{N}$ to mean a constant which may vary from line to line or even in the same line. If a constant $C$ depends on $r_1,r_2,\ldots$, we denote it by $C(r_1,r_2,\ldots)$. 
\end{itemize}

\section{Proof of Theorem \ref{pu2thm} and Theorem \ref{pu2thm2}}
For the rest of the article, we assume that $1<p<\infty$, $s\in I$ and $w\in W_p^{s}$ unless otherwise mentioned. Throughout this section, we assume $g$ is of the form $(g_1)$. First we obtain some useful results. Consider the energy functional $J_\la: W_0^{1,p}(\Omega,w) \to \mb R \cup \{\pm\infty\}$ defined by
\[J_\la(u)=\int_{\Om}G(x,\nabla u)- \la \int_{\Om}H(u)~dx  \]
where
$$
G(x,\nabla u)=\frac{1}{p}w(x)F(\nabla u)^p
$$
and
\begin{equation*}
H(t)=\left\{
\begin{aligned}
&\int_0^t h(\tau)\tau^{-\gamma}~d\tau,\; \text{if}\; t>0,\\
&0, \; \text{if}\; t\leq 0.
\end{aligned}\right.
\end{equation*}
Following Haitao \cite{YHaitao}, we establish the following result in the general weighted anisotropic setting.
\begin{Lemma}\label{Subsuplemma}
Let $2\leq p<\infty$. Suppose that $\underline{u}, \overline{u} \in W_0^{1,p}(\Omega,w) \cap  L^\infty(\Om)$ be weak subsolution and supersolution of \eqref{maineqn} respectively such that $0<\underline{u} \leq \overline{u}$ in $\Omega$ and $\uline{u}\geq c(\omega)>0$ for every $\omega \Subset  \Om$, for some constant $c(\omega)$. Then there exists a weak solution $u\in W_0^{1,p}(\Omega,w) \cap L^\infty(\Om)$ of \eqref{maineqn} satisfying $\uline{u}\leq u\leq \oline{u}$ in $\Om$.
\end{Lemma}
\begin{proof}
Let us consider the set
$$
S=\{v\in W_0^{1,p}(\Omega,w):\uline{u}\leq v\leq \oline{u}\text{ in }\Om\}.
$$
Since $\uline{u}\leq \oline{u}$ in $\Om$, we have $S\neq \emptyset$. We observe that $S$ is closed and convex. The result is proved in the following two Steps. \\
\textbf{Step $1$:} We claim that $J_\la$ admits a minimizer $u$ over $S$.\\
To this end, we prove that $J_\la$ is weakly sequentially lower semicontinuous over $S$. Indeed, let $\{v_k\}_{k\in\mathbb{N}} \subset S$ be such that $v_k \rightharpoonup v$ weakly in $W_0^{1,p}(\Omega,w)$. Then by the hypothesis on $h$, we have
\[H(v_k) \leq \int_0^{\oline{u}} h(\tau)\tau^{-\gamma}~d\tau \leq \frac{h(\|\oline{u}\|_\infty)}{(1-\gamma)}\|\oline{u}\|_\infty^{1-\gamma}.\]
Therefore by the Lebesgue Dominated Convergence theorem and weak lower semicontinuity of norm, the claim follows. Hence, there exists a minimizer $u \in S$ of $J_\la$ that is $J_\la(u)= \inf\limits_{v\in S}J_\la(v)$.\\
\textbf{Step $2$:} Here, we prove that $u$ is a weak solution of \eqref{maineqn}.\\
Let $\phi \in C_c^{1}(\Om)$ and $\epsilon>0$. We define
\begin{equation*}
\eta_{\epsilon}=\left\{
\begin{aligned}
&\oline{u} \;\;\;\text{if} \; u+\epsilon \phi \geq \oline{u},\\
&u+\epsilon \phi  \;\;\;\text{if}\; \uline{u}\leq u+\epsilon \phi \leq \oline{u},\\
&\uline{u} \;\;\;\text{if} \; u+\epsilon \phi \leq \uline{u}.
\end{aligned}\right.
\end{equation*}
Observe that $\eta_{\epsilon}=u+\epsilon\phi-\phi^{\epsilon}+\phi_{\epsilon}\in S$, where $\phi^{\epsilon}= (u+\epsilon \phi -\oline{u})^+$ and $\phi_{\epsilon}= (u+\epsilon \phi -\uline{u})^-$. By Step $1$ above, since $u$ is a minimizer of $J_\la$, we have
\begin{equation}\label{eq1}
\begin{split}
0 & \leq \lim_{t \to 0} \frac{J_\la(u+t(\eta_{\epsilon}-u))- J_\la(u)}{t}\\
& = \lim_{t \to 0} \frac{\displaystyle\int_{\Omega}\{G(x,\nabla u+t\nabla(\eta_\epsilon -u))-G(x,\nabla u)\}~dx }{t} - \la\lim_{t \to 0} \frac{\displaystyle \int_\Om \{H(u+t(\eta_\epsilon -u))- H(u)\} ~dx}{t}\\
&= I-\la J \;\text{(say)}.
\end{split}
\end{equation}
 We observe that
 \[I = \int_\Om \nabla_{\xi}G(x,\nabla u)\nabla (\eta_\epsilon -u)~dx,\]
   \[J = \lim_{t\to 0} \int_\Om (\eta_\epsilon -u)(u+\theta t (\eta_\epsilon -u))^{-\gamma}h(u+\theta t (\eta_\epsilon -u))~dx,\; \text{for some}\; \theta \in(0,1).\]
If $(\eta_\epsilon -u)\geq 0$, Fatou's Lemma gives
\[J \geq \int_\Om {(\eta_\epsilon -u) u^{-\gamma}h(u)}~dx.\]
 On the otherhand, if $(\eta_\epsilon -u) <0$, then since $(\eta_\epsilon-u)\geq \epsilon\phi$, we get $\phi \leq 0$. So, in such case, we obtain
 \[\left| {(\eta_\epsilon -u){(u+\theta t (\eta_\epsilon -u))^{-\gamma}}h(u+\theta t (\eta_\epsilon -u))} \right|\leq {-(\eta_\epsilon -u){\uline{u}^{-\gamma}} h(||\oline{u}||_{\infty})}\leq {-\epsilon\phi {\uline{u}^{-\gamma}}h(||\oline{u}||_\infty)}\in L^1(\Om)\]
 since $\phi\in C_c^{1}(\Om)$ and $\uline{u}\geq c(\omega)>0,$ whenever $\omega\Subset\Om$.
By the Lebesgue Dominated Convergence theorem, we have
$$
J=\int_{\Om}{(\eta_\epsilon -u){u^{-\gamma}}h(u)}\,dx.
$$
Plugging the above estimates of $I$ and $J$ in \eqref{eq1}, we arrive at
\begin{equation}\label{eq2}
\begin{split}
&0\leq  \int_\Om \nabla_{\xi}G(x,\nabla u)\nabla (\eta_\epsilon -u)~dx - \la \int_\Om {(\eta_\epsilon -u){u^{-\gamma}}h(u)}~dx\\
& \implies \frac{1}{\epsilon}(Q^\epsilon -Q_\epsilon)\leq \int_\Om \nabla_{\xi}G(x,\nabla u)\nabla \phi~dx - \la \int_\Om {{u^{-\gamma}}h(u)}\phi~dx
\end{split}
\end{equation}
where
\[Q^\epsilon= \int_\Om  \nabla_{\xi}G(x,\nabla u)\nabla\phi^{\epsilon}~dx - \la \int_\Om {{u^{-\gamma}}h(u)}\phi^{\epsilon}~dx \]
\[\text{and}\;Q_\epsilon= \int_\Om  \nabla_{\xi}G(x,\nabla u) \nabla \phi_\epsilon~dx - \la \int_\Om {{u^{-\gamma}}h(u)}\phi_\epsilon~dx .\]
Now we estimate $Q^\epsilon $ and $Q_\epsilon$ separately. We have
\begin{equation}\label{queps}
\begin{split}
\frac{1}{\epsilon}Q^\epsilon &\geq \frac{1}{\epsilon} \int_\Om \big\{\nabla_{\xi}G(x,\nabla u)-\nabla{\xi}G(x,\nabla \oline{u})\}\nabla \phi^\epsilon~dx\\
&\quad+ \frac{\la}{\epsilon}\int_\Om \oline{u}^{-\gamma}{f(\oline{u})}\phi^\epsilon~dx-\frac{\la}{\epsilon}\int_\Om {u}^{-\gamma}{f({u})}\phi^\epsilon~dx\\
&=\frac{1}{\epsilon}\int_{\Om^{\epsilon}}\big\{\nabla_{\xi}G(x,\nabla u)-\nabla_{\xi}G(x,\nabla \oline{u})\}\nabla(u-\oline{u})\,dx\\
& \quad \quad +\int_{\Om^{\epsilon}}\big\{\nabla_{\xi}G(x,\nabla u)-\nabla_{\xi}G(x,\nabla \oline{u})\}\nabla\phi\,dx
+\frac{\la}{\epsilon}\int_{\Om}f(u)(\oline{u}^{-\gamma}-u^{-\gamma})\phi^{\epsilon}\,dx\\
&\geq \int_{\Om^{\epsilon}}\big\{\nabla_{\xi}G(x,\nabla u)-\nabla_{\xi}G(x,\nabla \oline{u})\}\nabla\phi\,dx+\frac{\la}{\epsilon}\int_{\Om^{\epsilon}}f(u)(\oline{u}^{-\gamma}-u^{-\gamma})(u-\oline{u})\,dx\\
&\quad \quad+\la\int_{\Om^{\epsilon}}f(u)(\oline{u}^{-\gamma}-u^{-\gamma})\phi\,dx\\
&\geq o(1)
\end{split}
\end{equation}
using Lemma \ref{alg}, $\oline{u}$ is a weak supersolution of \eqref{maineqn}, $u\leq\oline{u}$ and $\displaystyle\int_{\Om^{\epsilon}}f(u)(\oline{u}^{-\gamma}-u^{-\gamma})\phi\,dx\leq{2{c(\omega)^{-\gamma}}f(||\oline{u}||_{\infty})}||\phi||_{\infty}<+\infty$, where $\Om^{\epsilon}= \text{supp}\;\phi^{\epsilon}$ and $\omega=\mathrm{supp}\,\phi$. Next, we observe that
\begin{equation}\label{qleps}
\begin{split}
\frac{1}{\epsilon}Q_\epsilon &\leq -\frac{1}{\epsilon}\int_{\Om_{\epsilon}}\nabla_{\xi}G(x,\nabla u)\nabla(u+\epsilon \phi -\uline{u})~dx\\
&+\frac{1}{\epsilon}\int_{\Om_{\epsilon}}\nabla_{\xi}G(x,\nabla\uline{u})\nabla(u+\epsilon\phi-\uline{u})\,dx+\frac{\la}{\epsilon}\int_{\Om}\uline{u}^{-\gamma}{f(\uline{u})}\phi_{\epsilon}\,dx-\frac{\la}{\epsilon}\int_{\Om}u^{-\gamma}{f(u)}\phi_{\epsilon}\,dx\\
&\leq \int_{\Om_{\epsilon}}\big\{\nabla_{\xi}G(x,\nabla\uline{u})-\nabla_{\xi}G(x,\nabla{u})\big\}\nabla\phi\,dx-\frac{\la}{\epsilon}\int_{\Om_{\epsilon}}f(u)(\uline{u}^{-\gamma}-u^{-\gamma})(u-\uline{u})\,dx\\
& \quad \quad-\la\int_{\Om_{\epsilon}}f(u)(\uline{u}^{-\gamma}-u^{-\gamma})\phi\,dx\\
&\leq o(1)
\end{split}
\end{equation}
using Lemma \ref{alg}, $\uline{u}$ is a weak subsolution of \eqref{maineqn}, $u \geq \uline{u}$ and $\displaystyle \int_{\Om_{\epsilon}}f(u)(\uline{u}^{-\gamma}-u^{-\gamma})\phi\,dx\leq {2c(\omega)^{-\gamma} f(\|\oline{u}\|_\infty)}\|\phi\|_\infty<+\infty$, where $\Om_\epsilon=\mathrm{supp}\,\phi_\epsilon$ and $\omega=\mathrm{supp}\,\phi$.

Using the estimates \eqref{queps} and \eqref{qleps} in \eqref{eq2}, we deduce that
\[0 \leq \int_\Om  \nabla_{\xi}G(x,\nabla{u})\nabla \phi~dx - \la \int_\Om f(u)u^{-\gamma}\phi~dx.\]
Since $\phi\in C_c^{1}(\Om)$ is arbitrary, our claim follows. This completes the proof.
\end{proof}

\begin{Remark}\label{Subsuplemmarmk}
If $F_{p,w}=\Delta_{p,w}$ or $S_{p,w}$ given by \eqref{ex}, then noting Lemma \ref{AI}, with the exact proof, Lemma \ref{Subsuplemma} holds for any $1<p<\infty$.
\end{Remark}


\begin{Lemma}\label{nlem}
Let $0<\gamma<1<p<\infty$ and $v_0\in W_0^{1,p}(\Om,w)$ be a weak solution of the problem 
\begin{equation}\label{neqn}
-F_{p,w}u=u^{-\gamma}\text{ in }\Om,\quad u>0\text{ in }\Om,\quad u=0\text{ on }\partial\Om.
\end{equation}
Then $v_0\in L^\infty(\Om)$.
\end{Lemma}
\begin{proof}
Let $k>1$, then by Remark \ref{tstrmk} choosing $\phi_k = (v_0-k)^+ \in W_0^{1,p}(\Om,w)$ as a test function in the equation (\ref{neqn}) and using H$\ddot{\text{o}}$lder's and Young's inequality with $\epsilon\in(0,1)$, we arrive at
\begin{align*}
\int_{\Omega}wF(\nabla\phi_k)^p\,dx\leq\,C(\epsilon)|A(k)|^\frac{p'}{q^{'}} + \epsilon\int_{\Omega}wF(\nabla\phi_k)^p\,dx,
\end{align*}
where $A(k)=\big\{x\in\Omega:v_0\geq k \text{ in }\Omega\big\}$. In the above estimate, we have also that $W_0^{1,p}(\Om,w)\to L^q(\Om)$ for some $q>p$ from Lemma \ref{emb}. Therefore, fixing $\epsilon\in(0,1)$, we obtain
$$
\int_{\Omega}wF(\nabla\phi_k)^p\,dx \leq C|A(k)|^\frac{p'}{q^{'}},
$$
where $C$ is some positive constant. Let $1 < k < h$, then since $A(h)\subset A(k)$, we have
\begin{align*}
(h-k)^p|A(h)|^\frac{p}{q}
\leq\Big(\int_{A(h)}(v_0-k)^{q}\,dx\Big)^\frac{p}{q}&\leq\Big(\int_{A(k)}(v_0-k)^{q}\,dx\Big)^\frac{p}{q}\\
&\leq C\int_{\Omega}wF(\nabla\phi_k)^p\,dx\leq C\,|A(k)|^\frac{p'}{q^{'}}.
\end{align*}
Therefore
$$
|A(h)| \leq\frac{C}{(h-k)^q}|A(k)|^\frac{q-1}{p-1}.
$$
Since $\frac{q-1}{p-1}>1$, by Lemma \ref{useful for L infinity estimate}, we have  
$
||v_0||_{L^\infty(\Omega)} \leq c,
$
where $c$ is a positive constant. Hence the result follows.
\end{proof}

\begin{Remark}\label{exrmk}
Noting Lemma \ref{regrmk} and Lemma \ref{alg}, we can apply \cite[Theorem 2.5]{Gpadm} to obtain the existence of $v_0\in W_0^{1,p}(\Om,w)$ solving the problem \eqref{neqn}.
\end{Remark}

\textbf{Proof of Theorem \ref{pu2thm}:} We construct a pair of weak subsolution and supersolution of \eqref{maineqn} according to Lemma \ref{Subsuplemma}. By Lemma \ref{evpthm}, there exists $e_1\in W_0^{1,p}(\Omega,w)\cap L^\infty(\Omega)$ such that
\begin{equation}\label{aevp}
-F_{p,w}e_1=\la_1 w e_1^{p-1} \; \text{in}\; \Om, \;\;e_1>0\text{ in }\Om,\;\; e_1=0\;\text{on}\; \partial \Om
\end{equation}
and for every $\omega\Subset\Omega$, there exists a positive constant $c(\omega)$ with $u\geq c(\omega)$ in $\omega$. {By $(h_2)$, we know that $\lim\limits_{t\to 0} t^{-\gamma}{h(t)}=\infty$, so we can choose $a_\la>0$ sufficiently small such that
\begin{equation}\label{ala}
\la_1 (a_\la e_1)^{p-1} \leq \la {(a_\la e_1)^{-\gamma}}{h(a_\la e_1)}.
\end{equation}
Let $\uline{u}=a_{\la}e_1$, then $\uline{u}\in W_0^{1,p}(\Omega,w)\cap L^\infty(\Omega)$ and by \eqref{aevp} and \eqref{ala}, we get
\begin{equation}\label{sub}
-F_{p,w}{\uline{u}} \leq \la {(a_\la e_1)^{-\gamma}}{h(a_\la e_1)}=\la {\uline{u}^{-\gamma}}{h(\uline{u})}\;\text{in}\; \Om.
\end{equation}
By Lemma \ref{nlem} and Remark \ref{exrmk}, there exists $v_0\in W_0^{1,p}(\Omega,w)\cap L^\infty(\Om)$ such that for every $\omega\Subset\Omega$ there exists a positive constant $c(\omega)$ satisfying $v_0\geq c(\omega)>0$ in $\omega$ and
\begin{equation}\label{fs}
-F_{p,w}v_0 = v_0^{-\gamma},\; v_0>0 \; \text{in}\; \Om,\; v_0=0\;\text{on}\; \partial\Om.
\end{equation}
By the hypothesis $(h_2)$, since $\lim\limits_{t\to \infty}t^{-(\gamma+p-1)}{h(t)}=0$, we choose $b_\la>0$ sufficiently large such that
\begin{equation}\label{bla}
{(b_\la \|v_0\|_\infty)^{-(\gamma+p-1)}}{h({b_\la \|v_0\|_\infty)}}\leq \frac{1}{\la \|v_0\|^{\gamma+p-1}_\infty}.
\end{equation}
We define $\oline{u}:= b_\la v_0$. Then $\oline{u}\in W_0^{1,p}(\Omega,w)\cap L^\infty(\Omega)$ and using \eqref{fs} and \eqref{bla}, we have
\begin{equation}\label{sup}
-F_{p,w}{\oline{u}} = v_0^{-\gamma}{b_\la^{p-1}}  \geq \la{(b_\la v_0)^{-\gamma}}{h({b_\la \|v_0\|_\infty)}} \geq \la{\oline{u}^{-\gamma}} {h(\oline{u})}\; \text{in}\; \Om,
\end{equation}
where we have also used the nondecreasing property of $h$ from $(h_1)$.
Thus, from \eqref{sub} and \eqref{sup}, it follows that $\uline{u}$ and $\oline{u}$ are weak subsolution and supersolution of \eqref{maineqn} respectively and the constants $a_\la, b_\la$ can be chosen in such a way that $\uline{u}\leq \oline{u}$. Therefore, by Lemma \ref{Subsuplemma}, the result follows.

\textbf{Proof of Theorem \ref{pu2thm2}:} Noting Lemma \ref{AI} along with Remark \ref{Subsuplemmarmk}, Remark \ref{evpthmrmk} and then proceeding along the lines of the proof of Theorem \ref{pu2thm}, the result follows.

\section{Proof of Theorem \ref{mthm} and Theorem \ref{mthm2}}
In this section, we consider the equation \eqref{maineqn} when $g$ is of the form $(g_2)$, which reads as
\begin{equation}\label{psp}
\begin{split}
-F_{p,w}u&=\lambda u^{-\gamma}+u^q\text{ in }\Omega,\quad u>0\text{ in }\Omega,\quad u=0\text{ on }\partial\Omega,
\end{split}
\end{equation}
where $\lambda>0$, $1<p<\infty$, $0<\gamma<1$, $w\in W_p^{s}$ for some $s\in I$ and $q\in(p-1,p_s^{*}-1)$. Here $p_s^{*}=\frac{Np_{s}}{N-p_{s}}$ if $1\leq p_s<N$, $p_s^{*}=\infty$ if $p_s\geq N$. 

First, we obtain some preliminary results. To this end, we define the energy functional $I_\la: W_0^{1,p}(\Omega,w)\to\mb R\cup\{\pm \infty\}$ corresponding to the problem \eqref{psp} by
\begin{equation}\label{fnl1}
I_\la(u) :=\frac{1}{p}\int_{\Omega}w(x)F(\nabla u)^p\,dx -\la \int_\Om \frac{(u^+)^{1-\gamma}}{1-\gamma}~dx -\frac{1}{q+1}\int_\Om (u^+)^{q+1}~dx.
\end{equation}
For $\epsilon>0$, we consider the approximated problem
\begin{equation}\label{apx}
\begin{aligned}
  -F_{p,w}u &=\la(u^{+} +\epsilon)^{-\gamma}+ (u^+)^q\;\text{in}\; \Om,\quad u=0 \; \text{ on }\; \partial\Om.
\end{aligned}
\end{equation}
We notice that the energy functional associated with the problem \eqref{apx} is given by
\begin{equation}\label{fnl2}
I_{\la,\epsilon}(u) = \frac{1}{p}\int_{\Omega}w(x)F(\nabla u)^p\,dx -\la\int_\Om \frac{[(u^+ +\epsilon)^{1-\gamma}-\epsilon^{1-\gamma}]}{1-\gamma}~dx -\frac{1}{q+1}\int_\Om (u^+)^{q+1}~dx.
\end{equation}
We observe that $I_{\la,\epsilon}\in C^1\big(W_0^{1,p}(\Omega,w),\mb R\big)$, $I_{\la,\epsilon}(0)=0$ and $I_{\la,\epsilon}(v)\leq I_{0,\epsilon}(v)$, for all $ v \in W_0^{1,p}(\Omega,w)$. Let us define
\begin{equation}\label{l}
l=
\begin{cases}
p_s^{*}=\frac{Np_s}{N- p_s},\text{ if } 1\leq p_s<N,\\
r,\text{ if }p_s\geq N,
\end{cases}
\end{equation}
where $r>1$ is such that $p-1<q<r-1$ if $p_s\geq N$. Next we prove that $I_{\la,\epsilon}$ satisfies the Mountain Pass Geometry.
\begin{Lemma}\label{MP-geo}
Let $2\leq p<\infty$. Then there exists $R>0,\,\rho>0$ and $\Lambda>0$ depending on $R$ such that
$$
\inf\limits_{\|v\|\leq R}I_{\la,\epsilon}(v)<0\;\text{and}\;
\inf\limits_{\|v\|=R}I_{\la,\epsilon}(v)\geq \rho,\text{ for }\la\in(0,\Lambda).
$$
Moreover, there exists $T>R$ such that
$
I_{\la,\epsilon}(Te_{1})<-1$ for $\la\in (0,\La)$, where $e_1$ is given by Lemma \ref{evpthm}.
\end{Lemma}
\begin{proof}
Recalling the definition of $l$ from \eqref{l}, we define $\theta=|\Om|^{\frac{1}{\left(\frac{l}{q+1}\right)'}}$. By H\"{o}lder's inequality and Lemma \ref{emb}, for every $v\in W_0^{1,p}(\Omega,w)$, we have
\begin{equation}\label{MP1}
\int_\Om (v^+)^{q+1}~dx \leq \left( \int_\Om |v|^{l}\right)^{\frac{q+1}{l}} |\Om|^{\frac{1}{(\frac{l}{q+1})'}}\leq C\theta\|v\|^{q+1},
\end{equation}
for some positive constant $C$ independent of $v$. Since
$$
\lim_{t\to 0}\frac{I_{\la,\epsilon}(te_1)}{t}=-\la\int_{\Omega}\epsilon^{-\gamma}e_{1}\,dx<0,
$$
we choose $k\in(0,1)$ sufficiently small and set $\|v\|=R :=k(\frac{q+1}{pC\theta})^\frac{1}{q+1-p}$ such that 
$$
\inf\limits_{\|v\|\leq R}I_{\la,\epsilon}(v)<0.
$$
Moreover, using the fact $R<(\frac{q+1}{pC\theta})^\frac{1}{q+1-p}$ and the estimate \eqref{MP1}, we have
\begin{equation}\label{up}
I_{0,\epsilon}(v)\geq \frac{R^p}{p}-\frac{C\theta R^{q+1}}{q+1}
:=
2\rho\,(\text{say})>0.
\end{equation}
We define $$\Lambda:=\frac{\rho}{\sup\limits_{\|v\|=R} \left(\displaystyle\frac{1}{1-\gamma}\int_\Om |v|^{1-\gamma}~dx \right)},$$
which is positive. Note that, since $\rho,R$ depends on $k,q,p,|\Omega|$ and $C$, so does $\Lambda$. We observe that
\begin{equation}\label{knownfact}
(v^{+}+\epsilon)^{1-\gamma}-\epsilon^{1-\gamma}\leq (v^+)^{1-\gamma}.
\end{equation}
Therefore, we have
\begin{align*}
I_{\lambda,\epsilon}(v)&\geq \frac{1}{p}\int_{\Om}w(x)F(\nabla v)^p\,dx-\frac{1}{q+1}\int_{\Om}(v^{+})^{q+1}\,dx-\frac{\la}{1-\gamma}\int_{\Om}(v^{+})^{1-\gamma}\,dx\\
&= I_{0,\epsilon}(v)-\frac{\la}{1-\gamma}\int_{\Om}(v^{+})^{1-\gamma}\,dx.
\end{align*}
Hence, using \eqref{up}, for $\la\in(0,\Lambda)$, we get
\begin{align*}
\inf\limits_{\|v\|=R} I_{\la,\epsilon}(v)&\geq\inf\limits_{\|v\|=R}I_{0,\epsilon}(v)-\la \sup\limits_{\|v\|=R} \left(\frac{1}{1-\gamma}\int_\Om |v|^{1-\gamma}~dx \right)\\
&\geq 2\rho -\la \sup\limits_{\|v\|=R} \left(\frac{1}{1-\gamma}\int_\Om |v|^{1-\gamma}~dx \right)\geq \rho.
\end{align*}
Finally, we observe that $I_{0,\epsilon}(te_1) \to -\infty$, as $t\to +\infty$. This gives the existence of $T>R$ such that $I_{0,\epsilon}(Te_1)<-1$. Therefore, 
\[I_{\la,\epsilon}(Te_1)\leq I_{0,\epsilon}(Te_1)<-1,\]
which completes the proof.
\end{proof}


Next, we prove that $I_{\la,\epsilon}$ satisfies the Palais Smale $(PS)_c$ condition.

\begin{Lemma}\label{PS-cond}
Let $2\leq p<\infty$. Then $I_{\la,\epsilon}$ satisfies the $(PS)_c$ condition, for any $c \in \mb R$, that is if $\{u_k\}_{k\in\mathbb{N}}\subset W_0^{1,p}(\Omega,w)$ is a sequence such that
\begin{equation}\label{PS1}
I_{\la,\epsilon}(u_k)\to c \; \text{and}\; I_{\la,\epsilon}^\prime(u_k) \to 0
\end{equation}
as $k \to \infty$, then $\{u_k\}_{k\in\mathbb{N}}$ contains a strongly convergent subsequence in $W_0^{1,p}(\Omega,w)$.
\end{Lemma}
\begin{proof}
We prove the result in two steps below.\\
\textbf{Step $1$.} First, we claim that if $\{u_k\}_{k\in\mathbb{N}} \subset W_0^{1,p}(\Omega,w)$ satisfies \eqref{PS1} then $\{u_k\}_{k\in\mathbb{N}}$ is uniformly bounded in $W_0^{1,p}(\Omega,w)$. To this end, by \eqref{knownfact}, for some positive constant $C$ (independent of $k$), we have
\begin{equation}\label{PS2}
\begin{split}
&I_{\la,\epsilon}(u_k)- \frac{1}{q+1}I_{\la,\epsilon}^\prime(u_k)u_k = \left( \frac{1}{p}-\frac{1}{q+1}\right)
\int_{\Omega}w(x)F(\nabla u_k)^p\,dx -{\la}\int_\Om \frac{(u_k^+ +\epsilon)^{1-\gamma}-\epsilon^{1-\gamma}}{1-\gamma}~dx\\
& \quad +\frac{\la}{q+1}\int_\Om (u_k^+ +\epsilon)^{-\gamma}u_k~dx\\
& \geq \left( \frac{1}{p}-\frac{1}{q+1}\right)\|u_k\|^p-C\|u_k\|^{1-\gamma},
\end{split}
\end{equation}
for some positive constant $C$ (independent of $k$), where we have also used Lemma \ref{emb} and H\"older's inequality. 
Noting $q+1>p$ and using \eqref{PS2}, we obtain
\begin{equation}\label{PS2-new1}
I_{\la,\epsilon}(u_k)- \frac{1}{q+1}I_{\la,\epsilon}^\prime(u_k)u_k \geq C_1\|u_k\|^p -C\|u_k\|^{1-\gamma},
\end{equation}
for some positive constants $C,C_1$ (independent of $k$). Using \eqref{PS1}, for $k$ large enough, we have
\begin{equation}\label{PS3}
\left| I_{\la,\epsilon}(u_k)- \frac{1}{q+1}I_{\la,\epsilon}^\prime(u_k)u_k\right| \leq C+o(\|u_k\|),
\end{equation}
for some positive constant $C$ (independent of $k$). Since $p>1$, combining \eqref{PS2-new1} and \eqref{PS3}, our claim follows. \\
\textbf{Step $2$.} We claim that upto a subsequence, $u_k \to u_0$ strongly in $W_0^{1,p}(\Omega,w)$ as $k \to \infty$.\\
By Step $1$, since $\{u_k\}_{k\in\mathbb{N}}$ is uniformly bounded in $W_0^{1,p}(\Omega,w)$, due to the reflexivity of $W_0^{1,p}(\Omega,w)$, there exists $u_0\in W_0^{1,p}(\Omega,w)$ such that upto a subsequence, $u_k \rightharpoonup u_0$ weakly in $W_0^{1,p}(\Omega,w)$ as $k \to \infty$. Again, by \eqref{PS1}, we have
\[\lim_{k\to \infty}\left(\mc \int_{\Om}w(x)F(\nabla u_k)^{p-1}\nabla _{\xi}F(\nabla u_k)\nabla u_0\,dx - \la \int_\Om (u_k^+ +\epsilon)^{-\gamma}u_0~dx - \int_\Om (u_k^+)^{q} u_0~dx\right)=0\]
and
\[\lim_{k\to \infty}\left(\mc \int_{\Om}w(x)F(\nabla u_k)^{p-1}\nabla _{\xi}F(\nabla u_k)\nabla u_k\,dx - \la \int_\Om (u_k^+ +\epsilon)^{-\gamma}u_k~dx - \int_\Om (u_k^+)^q u_k~dx\right)=0,\]
which gives
\begin{equation}\label{PS4}
\begin{split}
&\lim\limits_{k\to\infty}\int_{\Omega}w(x)\big\{F(\nabla u_k)^{p-1}\nabla_{\xi}{F}(\nabla u_k)-{F}(\nabla u_0)^{p-1}\nabla_{\xi}{F}(\nabla u_0)\big\}\nabla(u_k-u_0)\,dx\\
&=\lim\limits_{k\to\infty} \left( \la \int_\Om (u_k^+ +\epsilon)^{-\gamma}u_k~dx + \int_\Om (u_k^+)^q u_k~dx - \la \int_\Om (u_k^+ +\epsilon)^{-\gamma}u_0~dx - \int_\Om (u_k^+)^q u_0~dx\right)\\
&\quad -\lim_{k\to \infty}\left(\int_\Om  w(x){F}(\nabla u_0)^{p-1}\nabla_{\xi}{F}(\nabla u_0)\nabla u_k~dx - \int_\Om w(x){F}(\nabla u_0)^p~dx\right).
\end{split}
\end{equation}
Since $u_k \rightharpoonup u_0$ weakly in $W_0^{1,p}(\Omega,w)$ as $k \to \infty$, we get
\begin{equation}\label{PS5}
\lim_{k\to \infty}\left(\int_\Om w(x){F}(\nabla u_0)^{p-1}\nabla_{\xi}{F}(\nabla u_0)\nabla u_k~dx - \int_\Om w(x){F}(\nabla u_0)^p~dx\right)=0.
\end{equation}
On the otherhand, since
\begin{align*}
\left|(u_k^++\epsilon)^{-\gamma}u_0\right| \leq\epsilon^{-\gamma}u_0\text{ and }
\int_\Om \left|\epsilon^{-\gamma}u_0\right|dx \leq \epsilon^{-\gamma}\int_\Om|u_0|~dx< +\infty,
\end{align*}
by the Lebesgue Dominated convergence theorem, it follows that
\begin{equation}\label{PS6}
\lim_{k \to \infty} \int_\Om (u_k^+ +\epsilon)^{-\gamma}u_0~dx = \int_\Om (u_0^+ +\epsilon)^{-\gamma}u_0~dx.
\end{equation}
Since $u_k \to u_0$ pointwise almost everywhere in $\Om$ and for any measurable subset $E$ of $\Om$,
\begin{equation*}
\begin{split}
\int_E |(u_k^++\epsilon)^{-\gamma}u_k |~dx&\leq \int_E\epsilon^{-\gamma}|u_k|~dx\leq\|\epsilon^{-\gamma}\|_{L^\infty(\Om)}\|u_k\|_{L^{l}(\Om)}|E|^{\frac{l-1}{l}}\leq C(\epsilon)|E|^{\frac{l-1}{l}},
\end{split}
\end{equation*}
using Vitali's convergence theorem, we have
\begin{equation}\label{PS7}
\lim\limits_{k\to\infty} \la \int_\Om (u_k^+ +\epsilon)^{-\gamma}u_k~dx = \la \int_\Om (u_0^+ +\epsilon)^{-\gamma}u_0~dx.
\end{equation}
Since $q+1<l$, we have
\[\int_E |(u_k^+)^q u_0|~dx \leq \|u_0\|_{L^{l}(\Om)} \left(\int_E (u_k^+)^{ql^{'}}~dx\right)^{\frac{1}{l{'}}}\leq C_3 |E|^{\alpha}  \]
and
\[\int_E |(u_k^+)^q u_k|~dx \leq \|u_k\|_{L^{l}(\Om)} \left(\int_E (u_k^+)^{ql{'}}~dx\right)^{\frac{1}{l{'}}}\leq C_4 |E|^{\beta}  \]
for some positive constants $C_3,C_4,\alpha$ and $\beta$. Again using Vitali's convergence theorem, we get
\begin{equation}\label{PS8}
\lim_{k \to \infty} \int_\Om (u_k^+)^qu_0~dx  =\int_\Om (u_0^+)^qu_0~dx,
\end{equation}
and
\begin{equation}\label{PS9}
\lim_{k \to \infty} \int_\Om (u_k^+)^qu_k~dx  =\int_\Om (u_0^+)^qu_0~dx.
\end{equation}
Using \eqref{PS5}, \eqref{PS6}, \eqref{PS7}, \eqref{PS8} and \eqref{PS9} in \eqref{PS4}, we obtain
\[\lim\limits_{k\to\infty}\int_{\Omega}w(x)\big({F}(\nabla u_k)^{p-1}\nabla _{\xi}{F}(\nabla u_k)-{F}(\nabla u_0)^{p-1}\nabla _{\xi}{F}(\nabla u_0)\big)\nabla(u_k-u_0)\,dx =0.\]
Since $2\leq p<\infty$, using Lemma \ref{alg}, we obtain $u_k\to u_0$ strongly in $W_0^{1,p}(\Omega,w)$ as $k\to\infty$ which proves our claim.
\end{proof}

\begin{Remark}\label{multrmk}
Let $2\leq p<\infty$. Then by Lemma \ref{MP-geo}, Lemma \ref{PS-cond} and the Mountain Pass Lemma, for every $\la\in(0,\Lambda)$, there exists $\zeta_\epsilon \in W_0^{1,p}(\Omega,w)$ such that $I_{\lambda,\epsilon}^\prime(\zeta_\epsilon)=0$ {and}
$$
I_{\lambda,\epsilon}(\zeta_{\epsilon})=\inf_{\gamma\in\Gamma}\max_{t \in [0,1]}I_{\la,\epsilon}(\gamma (t)) \geq \rho >0,
$$
where 
$$
\Gamma =\big\{\gamma \in C([0,1],W_0^{1,p}(\Om,w)):\gamma(0)=0, \gamma(1)=Te_1\big\}.
$$
Moreover, as a consequence of Lemma \ref{MP-geo}, since for every $\la\in(0,\Lambda)$ we have $\inf\limits_{\|v\|\leq R} I_{\la,\epsilon}(v)<0$, by the weak lower semicontinuity of $I_{\la,\epsilon}$, there exists a nonzero $\nu_\epsilon\in W_0^{1,p}(\Om,w)$ such that $\|\nu_\epsilon\| \leq R$ and
\begin{equation}\label{limit-pass}
\inf\limits_{\|v\|\leq R} I_{\la,\epsilon}(v) =I_{\la,\epsilon}(\nu_\epsilon)<0<\rho \leq I_{\la,\epsilon}(\zeta_\epsilon).
\end{equation}
Thus, $\zeta_\epsilon$ and $\nu_\epsilon$ are two different non trivial critical points of $I_{\la,\epsilon}$, provided $\la\in(0,\Lambda)$.
\end{Remark}

\begin{Lemma}\label{non-negative}
Let $2\leq p<\infty$, then the critical points $\zeta_\epsilon$ and $\nu_\epsilon$ of $I_{\la,\epsilon}$ are nonnegative in $\Omega.$
\end{Lemma}

\begin{proof}
Let $u=\zeta_\epsilon$ or $\nu_\epsilon$. Therefore, since the integrand
$
\la(u^+ +\epsilon)^{-\gamma}+(u^+)^q
$
is nonnegative in $\Om$, testing \eqref{apx} with $v=\min\{u,0\}$, we get
$$
\int_{\Om}w(x){F}(\nabla v)^p\,dx=0.
$$
This proves that $v=0$ in $\Om$. Hence $u\geq 0$ in $\Om$, which completes the proof.
\end{proof}

\begin{Remark}\label{nonnegrmk}
If ${F}_{p,w}=\Delta_{p,w}$ or $S_{p,w}$ given by \eqref{ex}, then noting Lemma \ref{AI} and Remark \ref{evpthmrmk}, analogously it follows that Lemma \ref{PS-cond}, Remark \ref{multrmk} and Lemma \ref{non-negative} are valid for any $1<p<\infty$.
\end{Remark}

\begin{Lemma}\label{apriori}
Let $2\leq p<\infty$. Then there exists a constant $\Theta>0$ (independent of $\epsilon$) such that $\|v_\epsilon\| \leq \Theta$, where $v_\epsilon = \zeta_\epsilon$ or $\nu_\epsilon$.
\end{Lemma}
\begin{proof}
We notice that the result trivially holds if $v_\epsilon = \nu_\epsilon$. Thus, it is enough to deal with the case when $v_\epsilon= \zeta_\epsilon$. Recalling the terms from Lemma \ref{MP-geo} and Remark \ref{multrmk}, we define $A = \max\limits_{t \in [0,1]}I_{0,\epsilon}(tTe_1)$ then
\[A \geq \max_{t \in [0,1]} I_{\la,\epsilon}(tTe_1) \geq\inf_{\gamma\in\Gamma}\max_{t \in [0,1]}I_{\la,\epsilon}(\gamma (t)) = I_{\la,\epsilon}(\zeta_\epsilon)\geq \rho>0>I_{\la,\epsilon}(\nu_\epsilon).\]
Therefore
\begin{equation}\label{ap1}
\frac{1}{p}\int_{\Om}w(x){F}(\nabla\zeta_\epsilon)^p\,dx-{\la}\int_\Om \frac{(\zeta_\epsilon +\epsilon)^{1-\gamma}-\epsilon^{1-\gamma}}{1-\gamma}~dx -\frac{1}{q+1}\int_\Om \zeta_\epsilon^{q+1}~dx \leq A.
\end{equation}
Choosing $\phi=-\frac{\zeta_\epsilon}{p+1}$ as a test function in \eqref{apx} we obtain
\begin{equation}\label{ap2}
-\frac{1}{q+1}\int_{\Om}w(x){F}(\nabla\zeta_\epsilon)^p\,dx+\frac{\la}{q+1}\int_{\Om}\frac{\zeta_\epsilon}{(\zeta_{\epsilon}+\epsilon)^{\gamma}}\,dx+\frac{1}{q+1}\int_{\Om}\zeta_{\epsilon}^{q+1}\,dx=0.
\end{equation}
Adding \eqref{ap1} and \eqref{ap2} we have
\begin{align*}
\left(\frac{1}{p}-\frac{1}{q+1}\right)\int_{\Om}w(x){F}(\nabla\zeta_\epsilon)^p\,dx
 &\leq {\la}\int_\Om \frac{(\zeta_\epsilon +\epsilon)^{1-\gamma}-\epsilon^{1-\gamma}}{1-\gamma}~dx -\frac{\la}{q+1}\int_{\Om}\frac{\zeta_\epsilon}{(\zeta_\epsilon +\epsilon)^{\gamma}}\,dx+A\\
 & \leq C\int_\Om {\zeta_\epsilon} ^{1-\gamma}+A\leq C\|\zeta_\epsilon\|^{1-\gamma}+A,
\end{align*}
for some positive constant $C$ being independent of $\epsilon$, where we have used H\"older's inequality and Lemma \ref{emb}. Thus, since $q+1>p$, the sequence $\{\zeta_\epsilon\}$ is uniformly bounded in $W_0^{1,p}(\Omega,w)$ with respect to $\epsilon$. This completes the proof.
\end{proof}

\begin{Remark}\label{apriorinew}
If ${F}_{p,w}=\Delta_{p,w}$ or ${S}_{p,w}$ as given by \eqref{ex}, then noting Remark \ref{nonnegrmk}, with the same proof Lemma \ref{apriori} holds for any $1<p<\infty$.
\end{Remark}

\noi \textbf{Proof of Theorem \ref{mthm}:}
Let $2\leq p<\infty$. Then by Lemma \ref{non-negative} and Lemma \ref{apriori}, upto a subsequence, $\zeta_\epsilon \rightharpoonup \zeta_0$ and $\nu_\epsilon \rightharpoonup \nu_0$ weakly in $W_0^{1,p}(\Omega,w)$ as $\epsilon \to 0^+$, for some nonnegative $\zeta_0,\nu_0\in W_0^{1,p}(\Omega,w)$.\\
\textbf{Step $1$.} Let $v_0=\zeta_0$ or $\nu_0$. Here, we prove that $v_0\in W_0^{1,p}(\Omega,w)$ is a weak solution of the problem \eqref{psp}. Indeed, for any $\epsilon\in(0,1)$ and $t\geq 0$, we notice that
$$
{\la}{(t+\epsilon)^{-\gamma}}+t^q\geq {\la}{(t+1)^{-\gamma}}+t^q\geq \text{min}\left\{1,\frac{\la}{2}\right\}:=C>0, \text{ say}.
$$
Therefore, recalling that $v_\epsilon=\zeta_\epsilon$ or $\nu_\epsilon$, we have
$$
-{F}_{p,w} \,v_\epsilon={\la}{(v_\epsilon+\epsilon)^{-\gamma}}+v_\epsilon^q\geq C>0.
$$
Following the lines of the argument in \cite[Lemma 3.1]{BGM} and applying \cite[Corollary 3.59]{Juh}, we get the existence of $\xi\in W_0^{1,p}(\Omega,w)$ satisfying
$$
-{F}_{p,w}\xi=C\text{ in }\Omega,\,\,\xi>0\text{ in }\Omega,
$$
such that for every $\omega\Subset\Om$, there exists a constant $c(\omega)>0$ satisfying $\xi\geq c(\omega)>0$ in $\Om$. 
Then, for every nonnegative $\phi\in W_0^{1,p}(\Omega,w)$, we have
\begin{align*}
&\int_{\Om}w(x){F}(v_\epsilon)^{p-1}\nabla_{\xi}{F}(v_\epsilon)\nabla\phi\,dx=\int_{\Om}\Big({\la}{(v_\epsilon+\epsilon)^{-\gamma}}+v_\epsilon^q\Big)\phi\,dx\geq\int_{\Om}C\phi\,dx\\
&=\int_{\Om}w(x){F}(\nabla\xi)^{p-1}\nabla_{\xi}{F}(\nabla\xi)\nabla\phi\,dx.
\end{align*}
Testing with $\phi=(\xi-v_\epsilon)^+$ in the above estimate, we obtain
$$
\int_{\Om}w(x)\big\{{F}(\nabla\xi)^{p-1}\nabla_{\xi}{F}(\nabla\xi)-{F}(\nabla v_\epsilon)^{p-1}\nabla_{\xi}{F}(\nabla v_\epsilon)\big\}\nabla(\xi-v_\epsilon)^+\,dx\leq 0.
$$
Since $2\leq p<\infty$, by Lemma \ref{alg}, we have $v_\epsilon\geq \xi$ in $\Om$. Hence there exists a constant $c(\omega)>0$ (independent of $\epsilon$) such that
\begin{equation}\label{uniform}
v_\epsilon\geq c(\omega)>0,\text{ for every }\omega\Subset\Om. 
\end{equation}
Recalling the definition of $l$ from \eqref{l} along with Lemma \ref{apriori} and the fact \eqref{uniform}, by \cite[Theorem 2.16]{Mikko}, we have
$$
\int_{\Om}w(x){F}(\nabla v_0)^{p-1}\nabla_{\xi}{F}(\nabla v_0)\nabla \phi\,dx=\la\int_{\Om}{\phi}{v_0^{-\gamma}}(x)\,dx+\int_{\Om}v_0^{l}\phi\,dx,
$$
for every $\phi\in C_c^{1}(\Omega)$. Hence the claim follows.\\
\textbf{Step $2$.} Now we establish that $\zeta_0\neq \nu_0$. Choosing $\phi=v_\epsilon\in W_0^{1,p}(\Om,w)$ as a test function in \eqref{apx}, we get
$$
\int_{\Om}w(x){F}(\nabla v_\epsilon)^{p}\,dx=\la\int_{\Om}{v_\epsilon}{(v_\epsilon +\epsilon)^{-\gamma}}\,dx+\int_{\Om}v_\epsilon^{q+1}\,dx.
$$
Since $q+1<l$, using Lemma \ref{emb} we obtain
\begin{equation}\label{nl}
\lim\limits_{\epsilon\to 0^+}\int_{\Om}(v_\epsilon)^{q+1}\,dx=\int_{\Om}v_0^{q+1}\,dx.
\end{equation}
Moreover, since
$$
0\leq {v_\epsilon}{(v_\epsilon +\epsilon)^{-\gamma}}\leq v_\epsilon^{1-\gamma},
$$
using Vitali's convergence theorem, it follows that
$$
\la\lim\limits_{\epsilon\to 0^+}\int_{\Om}{v_\epsilon}{(v_\epsilon +\epsilon)^{-\gamma}}\,dx=\la\int_{\Om}v_0^{1-\gamma}\,dx.
$$
Therefore for every $\phi\in W_0^{1,p}(\Omega,w)$, we obtain
\begin{equation}\label{1}
\lim_{\epsilon\to 0^+}\int_{\Om}w(x){F}(\nabla v_\epsilon)^{p}\,dx=\la\int_{\Om}v_0^{1-\gamma}\,dx+\int_{\Om}v_0^{q+1}\,dx.
\end{equation}
By Remark \ref{tstrmk}, choosing $\phi=v_0$ as a test function in \eqref{psp} we get
\begin{equation}\label{2}
\int_{\Om}w(x){F}(\nabla v_0)^p\,dx=\la\int_{\Om}v_0^{1-\gamma}\,dx+\int_{\Om}v_0^{q+1}\,dx.
\end{equation}
Hence from \eqref{1} and \eqref{2}, we obtain
\begin{equation}\label{3}
\lim\limits_{\epsilon\to 0^+}\int_{\Om}w(x){F}(\nabla v_\epsilon)^p\,dx=\int_{\Om}w(x){F}(\nabla v_0)^p\,dx.
\end{equation}
Using Vitali's convergence theorem, we have
\begin{equation}\label{4}
\lim\limits_{\epsilon\to 0^+}\int_{\Om}[(v_\epsilon +\epsilon)^{1-\gamma}-\epsilon^{1-\gamma}]\,dx=\int_{\Om}v_0^{1-\gamma}\,dx.
\end{equation}
From \eqref{nl}, \eqref{3} and \eqref{4}, we have
$
\lim\limits_{\epsilon\to 0^+}I_{\la,\epsilon}(v_\epsilon)=I_{\la}(v_0),
$
which along with \eqref{limit-pass} gives $\zeta_0\neq \nu_0$.

\noi \textbf{Proof of Theorem \ref{mthm2}:} Noting Lemma \ref{AI}, Remark \ref{nonnegrmk}, Remark \ref{apriorinew} and Remark \ref{evpthmrmk}, proceeding similarly as in the proof of Theorem \ref{mthm}, the result follows.

\section{Appendix}
\textbf{Weighted anisotropic eigenvalue problem:} In this section, we consider an weighted anisotropic eigenvalue problem, which is crucial to obtain our main results above. Related eigenvalue problem in the absence of weight functions has been considered in \cite{BFK, Lind}. Eigenvalue problems for the weighted $p$-Laplace equation is studied in \cite{Drabek}. We study the following weighted ansiotropic eigenvalue problem,
\begin{equation}\label{evp}
-{F}_{p,w}u=\lambda a(x)|u|^{p-2}u\text{ in }\Omega,\quad u=0\text{ on }\partial\Omega,
\end{equation}
where $1<p<\infty$, $w\in W_{p}^{s}$ for some $s\in I$ and $a$ is a nonnegative measurable function in $\Omega$. Further, we assume that $a\in L^\frac{q}{q-p}(\Omega)$ for some $p<q<p_s^{*}$ where $p_s^{*}=\frac{Np_s}{N-p_s}$ if $1\leq p_s<N$ and $p_s^{*}=\infty$ if $p_s\geq N$. If $q=p$, we assume $a\in L^\infty(\Omega)$. Moreover, let $|x\in\Omega:a(x)>0|>0$. Below, we prove the result in case of $q>p$. In case of $q=p$, the proof is analogous.
\begin{Definition}\label{evdef}
We say that $\lambda\in\mathbb{R}$ is an eigenvalue of the problem \eqref{evp}, if there exists $u\in W_0^{1,p}(\Omega,w)\setminus\{0\}$ such that for every $\phi\in W_0^{1,p}(\Omega,w)$, we have
\begin{equation}\label{evpeqn}
\int_{\Omega}w(x){F}(\nabla u)^{p-1}\nabla_{\xi}{F}(\nabla u)\nabla\phi\,dx=\lambda\int_{\Omega}a(x)|u|^{p-1}u\phi\,dx.
\end{equation}
We define $u$ to be an eigenfunction of \eqref{evp} corresponding to the eigenvalue $\lambda$.
\end{Definition}
In this section our main result reads as follows:
\begin{Lemma}\label{evpthm}
Let $2\leq p<\infty$. Then there exists the least one eigenvalue $\lambda_1>0$ and at least one corresponding eigenfunction $e_1\in W_0^{1,p}(\Omega,w)\cap L^\infty(\Omega)\setminus\{0\}$ which is nonnegative in $\Om$ such that for every $\omega\Subset\Om$, there exists a positive constant $c(\omega)$ such that $e_1\geq c(\omega)>0$ in $\omega$. 
\end{Lemma}
\begin{proof}
\textbf{Existence and nonnegativity:} We claim that
$$
\lambda_1:=\inf\left\{\int_{\Omega}w(x){F}(\nabla u)^p\,dx:\int_{\Omega}a(x)|u|^p\,dx=1\right\}
$$
is the least eigenvalue of the problem \eqref{evp}. It is clear that $\lambda_1\geq 0$. Let $\{u_n\}_{n\in\mathbb{N}}\subset W_0^{1,p}(\Om,w)$ be a minimizing sequence for $\lambda_1$, which gives
\begin{equation}\label{evpmin}
\int_{\Omega}a(x)|u_n|^p\,dx=1\text{ and }\int_{\Omega}w(x){F}(\nabla u_n)^p\,dx=\lambda_1+\delta_n,
\end{equation}
where $\delta_n\to 0$ as $n\to\infty$. Therefore, $\{u_n\}_{n\in\mathbb{N}}$ is uniformly bounded in $W_0^{1,p}(\Omega,w)$ and so by the reflexivity of $W_0^{1,p}(\Omega,w)$ and Lemma \ref{emb}, upto a subsequence $u_n\rightharpoonup e_1$ weakly in $W_0^{1,p}(\Omega,w)$ and strongly in $L^q(\Omega)$. Now proceeding along the lines of the proof of \cite[Lemma $3.1$]{Drabek}, we obtain the existence of a nonnegative eigenfunction $e_1\in W_0^{1,p}(\Om,w)\setminus\{0\}$ such that $\int_{\Omega}a(x)|e_1|^p\,dx=1$ and 
\begin{equation}\label{l1c}
\lambda_1=\int_{\Omega}w(x){F}(\nabla e_1)^p\,dx
\end{equation}
is the eigenvalue of the problem \eqref{evp}.\\
\textbf{Uniform positivity:} Since $e_1\neq 0$, we have $\la_1>0$ and therefore noting Lemma \ref{regrmk} and Lemma \ref{alg}, applying \cite[Thoerem $3.59$]{Juh}, for every $\omega\Subset\Om$, there exists a positive constant $c(\omega)$ such that $e_1\geq c(\omega)>0$ in $\omega$. This also gives $e_1>0$ in $\Omega$.\\
\textbf{Boundedness:} Let $k>0$ and define by $u_k(x)=\min\{e_1(x),k\}$. For $m\geq 0$, choosing $\phi=u_k^{mp+1}$ as a test function in \eqref{evpeqn} we obtain
\begin{equation}\label{evptst}
\int_{\Omega}w(x){F}(\nabla e_1)^{p-1}\nabla_{\xi}{F}(\nabla e_1)\nabla u_k^{mp+1}\,dx=\lambda\int_{\Omega}a(x)e_{1}^{p-1}u_k^{mp+1}\,dx.
\end{equation}
By Lemma \ref{regrmk}, for some positive constant $c_1$, we observe that 
$$
\int_{\Omega}w(x){F}(\nabla e_1)^{p-1}\nabla_{\xi}{F}(\nabla e_1)\nabla u_k^{mp+1}\,dx\geq c_1(mp+1)\int_{\Omega}w(x)u_k^{mp}|\nabla u_k|^p\,dx.
$$
Now using the embedding result in Lemma \ref{emb} and proceeding with the same proof of \cite[Lemma $3.2$]{Drabek}, we obtain $e_1\in L^\infty(\Omega)$.
\end{proof}

\begin{Remark}\label{evpthmrmk}
If $F_{p,w}=\Delta_{p,w}$ or $S_{p,w}$ given by \eqref{ex}, then noting Lemma \ref{AI} and proceeding with the same arguments, Lemma \ref{evpthm} holds for any $1<p<\infty$.
\end{Remark}

\end{document}